\documentclass[12pt]{amsart}

\usepackage{amsbsy,amssymb,amscd,amsfonts,latexsym,amstext,delarray,amsmath,graphicx,color,caption,blkarray,mathtools,comment}
\usepackage{tikz}
\usetikzlibrary{cd}
\usetikzlibrary{matrix,arrows}
\usepackage{graphicx}
\usepackage{hyperref}
\usepackage[T1]{fontenc}
\input xy
\xyoption{all}
\usepackage{euscript}
\usepackage{multirow,charter}
\usepackage[normalem]{ulem}
\usepackage{MnSymbol}

\usepackage{amscd}
\usepackage{times,fancyhdr}
\usepackage{array}
\usepackage{epsfig}
\usepackage{graphicx}
\usepackage{color}
\usepackage{mathtools}
\usepackage{euscript}
\usepackage{footnote}
\usepackage[normalem]{ulem}

\makeatletter
\newcommand{\labitem}[2]{%
\def\@itemlabel{\textbf{#1}}
\item
\def\@currentlabel{#1}\label{#2}}
\makeatother

\newcounter{thecounter}
\numberwithin{thecounter}{section}
\newtheorem{lemma}[thecounter]{Lemma}
\newtheorem{proposition}[thecounter]{Proposition}
\newtheorem{theorem}[thecounter]{Theorem}

\newtheorem{corollary}[thecounter]{Corollary}
\newtheorem{conjecture}[thecounter]{Conjecture}

\theoremstyle{definition}
\newtheorem{defn}[thecounter]{Definition}

\DeclareMathOperator{\GL}{GL}

\DeclareMathOperator{\gl}{\mathfrak{gl}}
\DeclareMathOperator{\Aut}{Aut}

\DeclareMathOperator{\diag}{diag}
\DeclareMathOperator{\rank}{rank}

\DeclareMathOperator{\wts}{wts}
\DeclareMathOperator{\End}{End}
\DeclareMathOperator{\Exp}{Exp}

\DeclareMathOperator{\ad}{ad}

\DeclareMathOperator{\htt}{ht}
\DeclareMathOperator{\re}{{re}}
\DeclareMathOperator{\im}{{im}}

\DeclareMathOperator{\mult}{{mult}}
\DeclareMathOperator{\Span}{{span}}

\DeclareMathOperator*{\freeprod}{\raisebox{-0ex}{\scalebox{1.5}{$\Asterisk$}}}

\renewcommand{\a}{\alpha}
\renewcommand{\b}{\beta}
\newcommand{\g}{\gamma}

\newcommand{\Z}{{\mathbb Z}}

\newcommand{\C}{{\mathbb C}}
\newcommand{\N}{{\mathbb N}}
\newcommand{\Q}{{\mathbb Q}}

\newcommand{\fh}{{\mathfrak{h}}}
\newcommand{\fg}{{\mathfrak{g}}}
\newcommand{\fn}{{\mathfrak{n}}}

\newcommand{\Uhat}{{\widehat{U}}}
\newcommand{\Uhp}{{\widehat{U}^+}}

\renewcommand{\re}{{\text{\rm re}}}
\renewcommand{\im}{{\text{\rm im}}}
\newcommand{\fs}{\mathfrak{s}}

\newcommand{\n}{{\mathfrak{n}}}

\newcommand{\la}{\langle}
\newcommand{\ra}{\rangle}

\newcommand{\wt}{\widetilde}
\newcommand{\wh}{\widehat}

\usepackage{graphicx}
\usepackage{tikz}
\usetikzlibrary{decorations.fractals}
\usetikzlibrary{lindenmayersystems}
\usepackage{pgfplots}

\usepackage[matrix,arrow,ps]{xy}
 \UsePSspecials{dvips}

\usetikzlibrary{positioning}

\begin{document}

\title{\bf{Prosummability in Kac--Moody groups}}

\date{\today}

\begin{abstract} Let $\mathfrak{g}$ be a symmetrizable Kac--Moody algebra. 
We describe  {standard graded} $\mathfrak{g}$-modules $V$, which we use to construct a completion $\widehat{V}$ and pro-unipotent group $\widehat{U}$ in $\GL(\widehat{V})$.
These standard graded modules include the adjoint module, all integrable modules, Category~$\mathcal{O}$ modules, and opposite Category~$\mathcal{O}$ modules.  
We prove that the elements of $\widehat{U}$ are pro-summable series, that is, they are projective limits of summable series on quotients $\widehat{V}/\prod_{j=k}^\infty{V}_j$, for each $k>0$.  
We give an explicit construction of root subalgebras and their completions, corresponding to every root including the imaginary roots. 
We also construct complete root groups for imaginary roots, whose elements are also pro-summable series acting on $\widehat{V}$. 
We show that these groups are isomorphic to groups of power series in variables corresponding to basis elements for the imaginary root space. 
\end{abstract}

\author[Abid Ali]{Abid Ali}
\address{Department of Mathematics and Statistics, University of Saskatchewan
Saskatoon, SK S7N5E6, Canada }
\email{abid.ali@usask.ca}

\author[Lisa Carbone]{Lisa Carbone}
\address{Department of Mathematics, Rutgers University, Piscataway, NJ 08854-8019, USA}
\email{lisa.carbone@rutgers.edu\footnote{Corresponding author}}

\author[Elizabeth Jurisich]{Elizabeth Jurisich}
\address{School of Sciences, Mathematics and Engineering - College of Charleston, Charleston, SC 29403}
\email{jurisiche@cofc.edu}

\author[Scott H. Murray]{Scott H. Murray}
\address{Department of Mathematics, Rutgers University, Piscataway, NJ 08854-8019, USA}
\email{scotthmurray@gmail.com}

\thanks{The first author's research is partially supported by the Simons Foundation, Mathematics and Physical Sciences-Collaboration Grants for Mathematicians, Award Number 961267}
\maketitle

\section{Introduction}
Kac--Moody groups are Lie group analogs of  Kac--Moody algebras. 
There is an established theory of these groups (see, for example, \cite{MarT}, \cite{TitsSurvey}).
This theory is most developed in the affine case \cite{Garland}). 
Most constructions realize  Kac--Moody groups as automorphisms of the associated Kac--Moody algebras. 
Until recently, many constructions of Kac--Moody groups contained no explicit contribution from the imaginary root vectors (that is, the root vectors of the Lie algebra with  nonpositive squared norm). 

In order to describe the fundamental obstruction to describing imaginary root groups, we review the notion of \emph{summability}.  
An infinite sum $\sum_n E_n$ of operators on some space $V$ is called \emph{summable} if $\sum_n E_n(y)$ reduces to a finite sum for all $y\in V$ \cite{Bourbaki1-3,LL}.
Recall that, for a finite-dimensional semisimple Lie algebra $\frak g$, the simple root vectors $e_i$ and $f_i$ are $\ad$-nilpotent, that is, there exists $n$ such that
$(\ad(e_i))^{n}=(\ad(f_i))^{n}=0$.
Thus, 
$$\chi_{\a_i}(u):=\exp(u\ad(e_i))\quad\text{and}\quad\chi_{\a_i}(u):=\exp(u\ad(e_i))$$
are finite sums for $u\in\C$, where $\exp(X):
=1+t X+\frac{t^2}{2}X^2+\cdots$ as usual.
These elements generate the adjoint group $G_{\ad}$, whose elements are $\fg$-automorphisms.
Now consider $\frak g$ a symmetrizable Kac--Moody algebra. 
The simple root vectors $e_i$ and $f_i$ are \emph{locally $\ad$-nilpotent}, that is,
for all $y\in  \mathfrak{g}$ there exists $n=n(y)$  such that
$(\ad(e_i))^{n}(y)=(\ad(f_i))^{n}(y)=0$.
This means that $\chi_{\a_i}(u)$ and $\chi_{-\a_i}(u)$ are summable for $u\in\C$. 
This again ensures that
$\chi_{\pm\a_i}(u)$
can be viewed as an element of $\Aut(\mathfrak{g})$, and the group they generate is an adjoint Kac-Moody group $G_{\ad}$.
Given a representation $\rho$ acting on an integrable highest-weight $\fg$-module $V$, 
we can similarly define the Kac--Moody analog of a Chevalley group $G_V$ \cite{CG,Garland,KP}.



Obtaining a construction of Kac--Moody groups which fully incorporates the contribution from  imaginary root vectors has proved to be a challenging task.  
If $\a$ is a positive imaginary root with root vector $x_\alpha$, then $\exp\rho(x_\alpha)$ is not summable on the module~$V$.

In this work, we construct such Kac--Moody groups relative to a large  class of suitable modules. We work in the class of graded weight modules with a grading function $f:Q\to\Z$ satisfying conditions (F1)--(F4) of Lemma~\ref{L-f}, where $Q$ is the root lattice. We call such modules {\it standard graded modules}. This large class contains many representations $V$ of $\frak g$, including Category $\mathcal O$, opposite Category $\mathcal O$ and the adjoint representation. We recall that an irreducible highest weight module Category $\mathcal O$  $\frak g$-module is integrable if and only if the highest weight $\lambda$ is dominant integral (\cite{Kac74}, Section 10.1). Thus Category $\mathcal O$ contains many non-integrable representations. Nevertheless, our methods allow us to construct Kac--Moody groups with respect to these representations. This is discussed in more detail in Section~\ref{S-rtgp}.

In order to generalize previous constructions of  Kac--Moody groups, our goal is to obtain a group construction which has explicit elements corresponding to  imaginary root vectors. Our approach is to build an ambient structure where infinite sums of the form $\exp\ad(x_\alpha)$, for $\a$ imaginary, make sense. Inspired by the approach in \cite{Ku}, we form the positive formal completion $\widehat{\frak g}$ of $\frak g$:
$$\widehat{\frak g} = \frak n^- \oplus \frak h \oplus \widehat{\frak n}^+$$
That is, we pass from the  direct sum ${\frak n}^+ := \oplus_{\a\in\Delta^+} \frak g_\a$, (where elements are finite sums), to the infinite direct product
$$\widehat{\frak n}^+ := \prod_{\a\in\Delta^+} \frak g_\a = \prod_{k\in\mathbb N} \frak g_k, $$
(where elements are infinite sums), 
for a compatible $\Z$-grading, where each $\frak g_k$ is  finite dimensional.

We recall that a group is \emph{pro-unipotent}, respectively  \emph{pro-nilpotent},  if it is a projective limit of finite dimensional unipotent algebraic groups, respectively of finite dimensional nilpotent algebraic groups. Similarly a pro-nilpotent Lie algebra is a projective limit of finite dimensional nilpotent Lie algebras (see \cite{Ku} for more details.)

 In Section~\ref{prounipgp}, we construct a complete pro-unipotent group $\widehat{U}$ of automorphisms of a completion $\widehat{V}$ of a standard graded module $V$ for $\mathfrak g$.
 For these modules, $\exp\rho(x_\alpha)$ is pro-summable  when $x_\alpha\in \frak g_\alpha$ is imaginary, in the sense of  \cite{CJM} (see Proposition~\ref{P-formulas}).
 That is, each $\exp\rho(x_\alpha)$ is  summable on quotients $\widehat{V}/\prod_{j=k}^\infty{V}_j$, for each $k>0$.

Rousseau \cite{Rou16} made the constructions of complete  Kac--Moody groups by Mathieu \cite{Mat} and  \cite{MarThesis} more explicit. The construction in \cite{Rou16} is similar to the present one.

In Section~\ref{S-rtgp}, we define the root algebra $\mathfrak n_\a$, corresponding to an imaginary root $\a$, to be the smallest subalgebra of $\mathfrak g$ containing the root space $\mathfrak g_\a$.  This definition differs from the definition in \cite{MarT},  which defines a larger root algebra.  We also define  the completion, $\widehat{\mathfrak n}_\a$, to be the smallest complete subalgebra of $\widehat{\mathfrak g}$ containing $\mathfrak g_\a$.

 
We also define complete root groups $\wh{U}_\a$ in Section~\ref{S-rtgp} corresponding to  imaginary roots $\a$. We show that these are pro-unipotent groups of power series on a basis of the corresponding imaginary root space (see Theorem~\ref{T-Uhat}, (\ref{T-Uhat-free}) and that the elements of these groups are pro-summable series acting on $\widehat{V}$ (see Theorem~\ref{T-Uhat}, (\ref{CompleteProSum}).)  This definition  of $\wh{U}_\a$ differs from the imaginary root groups $\mathfrak{U}_{(\a)}$ constructed in \cite{MarT} in the following ways. 
Firstly, our group $\wh{U}_\a$ is defined relative to a general class of faithful representations, in contrast to the definition of  $\mathfrak{U}_{(\a)}$ in \cite{MarT} which is defined in terms of the adjoint representation. Secondly, our group $\wh{U}_\a$ is the smallest closed group containing $ \exp(\rho(\fg_{\a}))$ for each $\a$ (see Theorem~\ref{T-Uhat}), (3)). The group $\mathfrak{U}_{(\a)}$ in \cite{MarT} is larger as it contains contributions from integral multiples of the imaginary root $\a$. When $V$ is the adjoint representation, we have $\Uhat_\a\leq \mathfrak{U}_{(\a)}$ for each imaginary root $\a$. While our definitions are slightly different from those in  \cite{MarT} and \cite{Rou16}, the benefit of our approach is that we realize complete root groups for imaginary roots as Magnus groups (see also \cite{MarThesis}, Proposition 7.10 and \cite{MarT}, Example 8.89).  

We then  define
completions 
$$\widehat{B}_V:= \langle H_V, \widehat{U}_V \rangle \quad\text{and}\quad \widehat{G}_V:= \langle U^{-}_V, \widehat{B}_V \rangle $$
where  $H_V$ is a group associated to the Cartan subalgebra and $U^-_V$ is generated by exponentials corresponding to the negatives of the simple root vectors.
The group $\wh{G}_V$ differs from most other Kac--Moody group constructions in that it contains exponentials of positive imaginary root vectors. This is not the case in earlier definitions of Kac--Moody groups in this generality (such as, for example \cite{TitsJAlg} and \cite{CG}). Contributions of group elements corresponding to the imaginary roots  for affine Kac--Moody algebras over $p$-adic fields, expressed in coordinate forms, can be found in \cite[Section 7]{Ga} and \cite[Section 4]{BKP}. Our group construction involves root vectors of imaginary roots directly and more explicitly in general symmetrizable Kac--Moody groups. Furthermore, our definition of group elements extends to other fields of interest, including non-archimedian local fields, without significant difficulty. When $V$ is the adjoint representation of $\mathfrak g$, $\wh{G}_V$ is similar to the group constructed by Rousseau \cite{Rou16}. As in \cite{Rou16} and \cite{GR}, we expect  $\wh{G}_V$ to have an associated building with Coxeter group the Weyl group of the Kac--Moody group (see \cite{Rou16}, Prop. 3.16, Cor. 3.18).  It would be interesting to understand the  the role of group elements corresponding to imaginary root vectors from a geometric point of view  in the context of Bruaht-Tits buildings and hovels.

The authors are grateful to  Timoth\'ee Marquis, Shrawan Kumar and Guy Rousseau for helpful comments and for explaining their work.


\section{Kac--Moody algebras and Kac--Moody groups}\label{S-KMA}
In this section, we review the basic theory of Kac--Moody Lie algebras and Kac--Moody groups over the complex numbers $\C$. Our references for this section are \cite{Kac90}, \cite{Ku} and \cite{MP}.

\subsection{Generalized Cartan matrices}\label{SS-Cartan}
Let $I=\{1,2,\dots,\ell\}$ be a finite index set.
A \emph{generalized Cartan matrix} is
an  $\ell\times \ell$ matrix $A=(a_{ij})_{i,j\in I}$ with entries in $\Z$ satisfying
\begin{align*}
\tag{C1}\label{Cdiag} a_{ii}&=2,\\
\tag{C2}\label{Coffdiag} a_{ij}&\le0\quad\text{if $i\ne j$},\\
\tag{C3}\label{Czeros} a_{ji}&=0\quad\text{if $a_{ij}=0$},
\end{align*}
for all $i,j\in I$. 
The generalized Cartan matrix $A$ is \emph{symmetrizable} if there exist positive rational numbers $d_1,\dots, d_{\ell}$, such that
the matrix $\diag(d_1,\dots, d_\ell) A$ is symmetric.

We assume throughout that $A$ is a symmetrizable generalized Cartan matrix.
Let $\frak{h}$ be a $\C$-vector space and let $\langle\circ,\circ\rangle: \frak{h}^*\times\frak{h}\to \C$ denote the natural pairing,
that is, $\langle \varphi, h\rangle:= \varphi(h)$ for $\varphi\in\fh^*$, $h\in\fh$. 
As in \cite{Kac90}, $(\fh,\Pi,\Pi^\vee)$ is a \emph{realization} of $A$ if
\begin{description}
\item[\rm(R1)]\label{Rdim} $\fh$ has dimension $n:= 2\ell-\rank(A)$,
\item[\rm(R2)]\label{Rindep} $\Pi=\{\alpha_1,\dots,\alpha_{\ell}\}\subseteq \frak{h}^*$ and
$\Pi^\vee=\{\alpha_1^\vee,\dots,\alpha_{\ell}^\vee\}\subseteq \frak{h}$ are  linearly independent sets,
\item[\rm(R3)]\label{Rcartan} $\langle\alpha_j,\alpha_i^{\vee}\rangle=\alpha_j(\alpha_i^{\vee})=a_{ij}$ for $i,j\in I$.
\end{description}
We call the elements $\a_i$ \emph{simple roots} 
and $\a^\vee_i$ \emph{simple coroots}.
Every symmetrizable generalized Cartan matrix has a realization~\cite{Kac90}.

\subsection{Kac--Moody algebras}\label{SS-alg}
As in \cite{MP} and \cite[Theorem 9.11]{Kac90}, given a realization $(\fh,\Pi,\Pi^\vee)$,  the associated \emph{Kac--Moody algebra} $\frak{g}=\fg(A)$ is the Lie algebra over $\C$ with generating set  
$\frak{h}\cup\{e_i,f_i\mid i\in I\}$ and defining relations:
\begin{align*}
\tag{L1}\label{Lhh} [h,h']&=0;\\
\tag{L2}\label{Lhef} [h,e_i]&=\langle\alpha_i,h\rangle e_i;&
[h,f_i]&=-\langle\alpha_i,h\rangle f_i;\\
\tag{L3}\label{Lef} [e_i,f_i]&=\alpha_i^{\vee};&
[e_i,f_j]&=0;\\
\tag{L4}\label{Lad} (\ad e_i)^{-a_{ij}+1}(e_j)&=0;&
(\ad f_i)^{-a_{ij}+1}(f_j)&=0;
\end{align*}
for $h,h'\in\frak{h}$, and $i,j\in I$.

We say that $A$ is of \emph{finite type} if $A$ is positive definite, and that $A$ is of {\it affine type} if $A$ is positive semidefinite but not positive definite. If $A$ is not of finite or affine type, we say that $A$ is of \emph{indefinite type}.  
In particular, $A$ is of \emph{hyperbolic type} if $A$ is indefinite type but every proper indecomposable submatrix is either of finite or of affine type.

\subsection{Roots and the Weyl group}\label{RWG}
The \emph{roots}  of $\frak{g}$ are the nonzero  
$\a\in\fh^*$ for which the corresponding \emph{root space}
$$\frak{g}_{\alpha} := \{x\in \frak{g}\mid[h,x]=\alpha(h)x\text{ for all $h\in \frak{h}$}\}$$
is non-zero. 
Note that $0\in\fh^*$ is not considered a root, but $\fg_0=\fh$.
The simple roots $\a_i$ and their negatives have root spaces $\frak{g}_{\a_i}=\C e_i$ and $\frak{g}_{-\a_i}=\C f_i$.  
Every root $\alpha$ can be written in the form $\alpha=\sum_{i=1}^{\ell} k_i\alpha_i$ where the $k_i$ are integers, with either all $k_i \ge 0$, in which case $\alpha$ is called {\it positive}, or all $k_i\leq 0$, in which case $\alpha$ is called {\it negative}. 
The {\it height} of $\alpha$ is  $\htt(\alpha):=\sum_{i=1}^{\ell}k_i$ and the \emph{multiplicity} of $\a$ is $\mult(\a):= \dim(\fg_\a)$.
We denote the set of roots by $\Delta$ and the set of positive (respectively negative) roots by $\Delta^+$ (respectively $\Delta^-$). 
The \emph{root lattice} is $Q:=\Span_\mathbb{Z}(\Pi)\subseteq\frak h^*$.

The Lie algebra $\fg$ has a \emph{root  space decomposition}
$$ \fg=\fh\oplus\bigoplus_{\a\in\Delta}\fg_\a$$ 
and a \emph{triangular decomposition} $$\frak{g}=\frak{n}^- \oplus \frak{h} \oplus \frak{n}^+,$$ 
where 
$\frak{n}^+ =\bigoplus\limits_{\alpha\in\Delta^+}\frak{g}_{\alpha}$
and
$\frak{n}^- = \bigoplus\limits_{\alpha\in\Delta^-}\frak{g}_{\alpha}$, by \cite[Theorem 1.2]{Kac90}.

Since $A$ is  symmetrizable,  $\frak{g}$ admits a nondegenerate symmetric invariant bilinear form $(\circ,\circ)_{\fg}$ satisfying 
\begin{align*}
\tag{B1}\label{B-hh}
(\a_i^\vee,\a_j^\vee)_{\mathfrak{g}} &= d_ia_{ij},\\
\tag{B2}\label{B-ef}
({e}_{i},{f}_{j})_{\mathfrak{g}} &= \delta_{ij} , \\
\tag{B3}\label{B-gg}
\left(\fg_{\a}, \fg_{\b} \right)_{\fg}&=0 \quad\text{if $\a,\b\in\Delta$ with $\a+\b\ne0$,}\\
\tag{B4}\label{B-hg}
\left(\fh, \fg_{\a} \right)_{\fg}&=0 \quad\text{if $\a\in\Delta$}
\end{align*}  
for $i,j\in I$.
The restriction of $(\circ,\circ)_\fg$ to $\fh$ then induces a symmetric bilinear form $(\circ,\circ)=(\circ,\circ)_{\fh^*}$ on $\fh^*$  with
$$(\a_i,\a_j) = d_ia_{ij}.$$
Note that 
$$\langle\a_i,\a_j^\vee\rangle =\a_i(\a_j^\vee)=a_{ij} = \dfrac{2d_{i}a_{ij}}{d_{i}a_{ii}}= \dfrac{2(\a_i,\a_j)}{(\a_i,\a_i)}.$$
A root $\a\in\Delta$ is called \emph{real} (respectively \emph{imaginary}) if
$(\a,\a)>0$ (respectively $(\a,\a)\le0$).
We write $\Delta_\re$ (respectively $\Delta_\im$) for the set of all  real (respectively imaginary) roots.
We also write $\Delta_\re^{\pm}=\Delta_\re\cap\Delta^\pm$ and
$\Delta_\im^{\pm}=\Delta_\im\cap\Delta^\pm$.

Given a real root $\a$, we define the corresponding \emph{coroot} $\a^\vee$ to be the unique
element of $\fh$ such that
$$\langle x, \a^\vee\rangle = \frac{2(x,\a)}{(\a,\a)}$$
for all $x\in\fh^*$.  Note that this agrees with the definition of $\a_i^\vee$ above.
Each coroot $\alpha^\vee$ is an integral linear combination of the simple coroots, and so is an element of the \emph{coroot lattice} $Q^\vee:=\Span_\Z(\Pi^\vee)\subset\mathfrak h$. 
We define the \emph{reflection in  $\a$} by
$$w_\a: \quad \fh^*\rightarrow\fh^*, \quad  x\mapsto    
   x - \langle x,\alpha^{\vee}\rangle\alpha.$$
   
A \emph{simple reflection} is $w_i:= w_{\a_i}$ for $i\in I$.
The group $W\subseteq \GL(\frak{h}^{\ast})$ generated by the simple reflections is called the \emph{Weyl group}. 
The bilinear form $(\circ,\circ)$ on ${\frak{h}^*}$ is $W$-invariant. 
A root $\a$ is real if and only if there exists $w\in W$ such that $w\alpha$ is a
simple root, 
thus  $\Delta_{\re}=W\cdot \Pi$ \cite{Kac90}.  
Since $w\cdot \a=\a_i$ for some $w\in W$ and $i\in I$, we have $w_\a =w w_i w^{-1}\in W$. The group $W$ preserves the sets $\Delta_{\re}$, $\Delta^+_{\im}$,
and $\Delta_{\im}^-$.

Using the generalized Cartan matrix $A=(a_{ij})_{i,j\in I}$, we define the \emph{Coxeter matrix}
$(c_{ij})_{i,j\in I}$ by
$$
c_{ij} := 
 \begin{cases} 
  1 &\text{if $i=j$,}\\
  2 &\text{if $a_{ij}a_{ji}=0$,}\\
  3 &\text{if $a_{ij}a_{ji}=1$,}
 \end{cases}\qquad
 \begin{cases}
  4 &\text{if $a_{ij}a_{ji}=2$,}\\
  6 &\text{if $a_{ij}a_{ji}=3$,}\\
  0 &\text{otherwise.}
 \end{cases}
$$
The Weyl group $W$ is a \emph{Coxeter group} with presentation
$$W=\langle w_1,\dots,w_\ell \mid (w_iw_j)^{c_{ij}}=1 \text{ for $i,j\in I$}
\rangle.$$

For $w\in W$, a word $w=w_{i_1}w_{i_2}\cdots w_{i_k}$ of minimal length is called a \emph{reduced word} for $w$. We call $\ell(w):=k$ the \emph{length} of $w$.

\subsection{Weight modules}\label{SS-module} 
There are elements $\omega_i\in \fh^*$ satisfying $\langle \omega_i,\alpha_j^\vee\rangle=\delta_{ij}$ for $i,j\in I$. We call 
$\{\omega_{i}\mid i\in I\}$ a set of \emph{fundamental weights}.
An element $\mu\in\mathfrak h^*$ is called \emph{regular}, if $\langle \mu,\alpha_i^\vee\rangle\ne0$ for all $i\in I$. 


A subgroup $P$ of $\mathfrak h^*$ is called a \emph{weight lattice} if 
\begin{description}
    \item[\rm (W1)] $Q\subseteq P$,
    \item[\rm (W2)] $\langle \mu,\a_i^\vee \rangle \in\Z$ for all $\mu\in P$ and $i\in I$,
    \item[\rm (W3)] $P$ contains a weight which is minimal with respect to being regular,
    \item[\rm (W4)] $P$ contains a set of fundamental weights $\{\omega_i\mid i\in I\}$.
\end{description}
If the fundamental weights can be extended to a set
$$\{\omega_1,\dots,\omega_\ell,\omega_{\ell+1},\dots,\omega_n\}$$
that is both a $\Z$-basis for $P$ and a $\C$-basis for $\fh^*$, then we call $P$ a \emph{restricted} weight lattice. 
It is easily seen that there is a restricted weight lattice for every~$\fg$. 
Note that if $\det(A)\ne0$ then the weight lattice 
$$P = \{ \mu\in\fh^* \mid \langle \mu,\a_i^\vee \rangle \in\Z\text{ for all $i\in I$}\}$$
is unique,
and the set of fundamental weights is a $\Z$-basis of $P$.
We assume throughout that $P$ is a restricted weight lattice (see Section 6.1 of \cite{MP}).

An element $\lambda\in \frak{h}^{\ast}$   is \emph{dominant} if $\la \lambda, \alpha_i^{\vee}\ra\geq 0$, for all $i\in I$. We denote the set of dominant elements of $P$ 
by $P^{+}$ and write $Q^+=\sum_{i=1}^\ell \Z_{\ge0}\a_i$.
The \emph{dominance order} $\le $ on~$\frak{h}^{\ast}$ is defined by $\mu\le \lambda $ when $\lambda-\mu\in Q^+$.

Let $V$ be a $\frak{g}$-module with defining homomorphism 
$\rho:\frak g\to\gl(V)$.
The \emph{weight space} of~$\mu\in \frak{h}^\ast$ is
$$V_\mu:= \{v\in V\mid \rho(h)(v) = \mu(h) v \text{ for all $h\in\frak{h}$} \},$$
and the set of \emph{weights} of $V$ is
$$\wts(V):=\left\{\mu\in \frak h^\ast \mid V_{\mu}\neq 0\right\}.$$
A nonzero element of $V_\mu$ is called a \emph{weight vector} with \emph{weight} $\mu$.
 We call $V$ \emph{diagonalizable} if 
$$V=\bigoplus_{\mu\in\wts(V)} V_\mu.$$
A diagonalizable module $V$ whose weight spaces are all finite dimensional is called a \emph{weight module}.
Recall that $x\in\frak{g}$ acts \emph{locally nilpotently} on $V$ if, for each $v\in V$, there is a natural number $n=n(x,v)$ such that
$\rho(x)^{n}(v)=0$. 
We call $V$ an \emph{integrable module} if
it is a weight module, and 
$e_i$ and $f_i$ act locally nilpotently on~$V$ for all $i\in I$. 
If $\det(A)\ne0$ and $V$ is integrable, then $\wts(V)\subseteq P$ as in \cite{Kac90}.

The \emph{fundamental group} $P/Q$ is an abelian group which is infinite if $\det(A)=0$, and has order $|\det(A)|$ otherwise \cite{CS}.
The \emph{weight lattice of $V$} is $L_V:=\Span_\Z(\wts(V))$.
Further, $V$ is faithful if and only if $Q\subseteq L_V$.
Hence $L_V/Q$ is a subgroup of $P/Q$ when $V$ is a faithful integrable module. 


The \emph{adjoint module}  $V=\fg$ with $\rho=\ad$ is a  faithful integrable module with $\wts(V)=\Delta$.

We call $V$ a \emph{highest weight module} with  highest weight $\lambda$, if it has a weight vector $v_\lambda$ (the \emph{highest weight vector}) of weight $\lambda$ such that $\n^+ \cdot v_\lambda=0$ and the $\fn^-$-module generated by
$v_\lambda$ is all of $V$. The weight $\lambda$ 
satisfies $\mu\le \lambda$ for all $\mu\in\wts(V)$. 
For each $\lambda\in P^+$, there is a  unique irreducible highest weight module and this module is integrable \cite{Kac90}. 
More generally,  $V$ is a \emph{Category $\mathcal{O}$ module} \cite{MP95} if it is a weight module and there are finitely many weights $\lambda_1,\dots,\lambda_k\in P$ such that
every weight of $V$ is less than or equal to at least one of the $\lambda_i$. 

On the other hand, $V$ is a \emph{lowest weight module} with lowest weight $\lambda$ if it has a \emph{lowest weight vector} $v_\lambda$  such that $\fn^-\cdot v_\lambda =0$ and
the module generated by $\n^+ \cdot  v_\lambda$ is all of $V$.
More generally,  $V$ is an \emph{opposite Category $\mathcal{O}$ module} if it is a weight module and there are finitely many weights $\lambda_1,\dots,\lambda_k\in P$ such that
every weight of $V$ is greater than or equal to at least one of the $\lambda_i$. 


We assume throughout that $P$ is a  restricted weight lattice and $V$ is a  weight $\fg$-module, all of whose weights are contained in $P$. Thus the weights of $V$ are \emph{integral},
that is, $\la \mu, \alpha_i^{\vee}\ra\in \Z$ for all $\mu\in\wts(V)$ and~$i\in I$. 

\subsection{Kac--Moody Chevalley groups} 
Suppose $V$ is integrable. Then it follows that
\begin{align*}
\chi_{\a_i}(t)&:=\exp(t \rho(e_i))=1+t \rho(e_i)+\frac{t^2}{2}(\rho(e_i))^2+\cdots \quad\text{and}\\
\chi_{-\a_i}(t)&:=\exp(t \rho(f_i))=1+t \rho(f_i)+\frac{t^2}{2}(\rho(f_i))^2+\cdots
\end{align*}
are \emph{summable} for $t\in\C$, that is, they reduce to finite sums when applied to an element of $V$.
Hence they define elements of $\GL(V)$. 
We define the \emph{unipotent groups} $U^+_V$ and $U^-_V$ by
$$ U_V^\pm := \langle \chi_{\pm\a_i}(t)\mid i\in I,\, t\in\C\rangle.$$

Given $\varpi\in P$ and $t\in\C^\times$, we  define $h_{\varpi}(t)$  on  $V=\bigoplus_{\mu\in\wts(V)} V_\mu$ by
$$h_{\varpi}(t)\cdot v = t^{\langle\mu, \varpi\rangle} v $$
for $v\in V_\mu$.
Then $h_{\varpi}(t)$ is also in $\GL(V)$ and we define the \emph{toral group}
$$H_V=\langle h_{\varpi}(t)\mid \varpi\in L_V,\, t\in \C^\times\rangle.$$

Together these generate a {\it Kac--Moody Chevalley group}  
$$G_V:=\langle H_V,\ U^+_V,\, U^-_V\rangle.$$

For $i\in I$ and $s\in\C^\times$, let  
 \begin{eqnarray*}\label{rootref}
 \widetilde{w}_{i}(s)= \chi_{\alpha_i}(s)\chi_{-\alpha_i}(-s^{-1})\chi_{\alpha_i}(s),
\end{eqnarray*}
and set $\widetilde{w}_{i}:=\widetilde{w}_{i}(1)$. 
We write $\widetilde{W}$ for the subgroup of $G_V$ generated by $\widetilde{w}_{i}$ for all $i\in I$. Note that $\widetilde{W}H_V/H_V\cong W$   by the map  $\widetilde{w}_iH_V\mapsto w_i$. 
Each $w\in W$ can be written as  a reduced word $w_{j_{1}}w_{j_{2}}\dots w_{j_{m}}$ and we denote the corresponding element 
$
\widetilde{w}=\widetilde{w}_{j_{1}}\widetilde{w}_{j_{2}}\dots \widetilde{w}_{j_{m}} \in\widetilde{W}.
$
\begin{proposition} Each
$\widetilde{w}$ is independent of the choice of reduced word for $w$.
\end{proposition}
\begin{proof}
By \cite{Tits69} we can transform any reduced word for $w$ into any other reduced word for $w$ using only the braid relations
$w_i w_jw_i\cdots = w_j w_i w_j\cdots $
where both sides have length $c_{ij}$. 
These relations are only non-trivial for $c_{ij}=2,3,4,6$,
so it suffices to show that 
$\wt{w}_i \wt{w}_j \wt{w}_i\cdots = \wt{w}_j \wt{w}_i \wt{w}_j\cdots $
in subgroups of Cartan type $A_1A_1$, $A_2$, $B_2$, $G_2$, respectively.
This is \cite[Proposition~9.3.2]{Springer}.
\end{proof}
\begin{lemma}[\cite{Kac90}, Lemma~3.8(a)]\label{ww}  We have 
$\widetilde{w}\cdot V_\mu=V_{w\mu}$ and so $\dim V_{\mu}=\dim V_{w\mu}$. 
\end{lemma}

\begin{corollary} If $V$ is the adjoint representation $V=\fg$, we have $\wts(V)=\Delta$ and $V_\a=\fg_\a$ for $\a\in\Delta$. Hence 
$\widetilde{w}\cdot \fg_\a=\fg_{w\a}$ and so $\mult(\a)=\mult(w\a)$. 
\end{corollary}

For each $\a\in\Delta^{\re}$, 
we fix $w\in W$   such that $\alpha=w\alpha_i$, and write $x_\a = \widetilde{w}e_i$.
Now the root space $\fg_\a = \wt{w}\cdot \C e_i$, and we fix the basis element $x_\a = \widetilde{w}e_i$ so that 
$\fg_\a =  \C x_\a$.
Define
$\chi_{\a}(t):=\exp(t \rho(x_\a))$.
Then
$$\chi_{\alpha}(t)=\exp(t\rho( \widetilde{w} e_i))=\widetilde{w}\chi_{\alpha_i}( t)\widetilde{w}^{-1}\in G_V.$$ 
The \emph{root group} of a real root $\a$ is
$$U_\a= \{\chi_\a(t)\mid t \in \C\}\cong \C^+,$$
where $\C^+$ denotes the additive group of $\C$.
We consider imaginary root groups in Section~\ref{S-rtgp}.

\begin{theorem}[\cite{CER}]\label{T-indep} 
For faithful integrable modules $V$, $U_V^+$ is independent of the choice of $V$, while  $G_V$ depends only on~$L_V$.
\end{theorem}
Hence we write $U^\pm$ for $U_V^\pm$, and $G_L$ for $G_V$ when $L=L_V$.

\section{Completions of modules and pro-summability}\label{S-prosum}
In this section, we construct completions of Kac--Moody algebras and their modules, and then define pro-summable maps (Subsection \ref{SS-pro-sum}) on these completions.
Let $I=\{1,\dots,\ell\}$ be a finite index set, let $A=(a_{ij})_{i,j\in I}$ be a symmetrizable generalized Cartan matrix, let $\fg=\fg(A)$ be the Kac--Moody algebra of $A$,
and let $P$ be a restricted weight lattice of $\fg$.
Let $V$ be  a weight module of $\fg$ with $\wts(V)\subseteq P$.

\subsection{Graded algebras and modules}\label{SS-graded}
First we review the general theory of graded Lie algebras and modules \cite{NijRich}.
Fix a lattice  $L$.
Recall that a Lie algebra $\mathfrak{L}$ over $\C$ is \emph{L-graded} if
$$\mathfrak{L} =\bigoplus_{a\in L} \mathfrak{L}_a,$$
with $\left[\mathfrak{L}_a,\mathfrak{L}_b\right]\subseteq \mathfrak{L}_{a+b}$  and $\mathfrak{L}_a$ is finite dimensional for all $a,b\in L$.
A subalgebra $\mathfrak{M}\subseteq \mathfrak{L}$ is a \emph{graded subalgebra} if
$$\mathfrak{M} =\bigoplus_{a\in L} \mathfrak{M}_a$$
where $\mathfrak{M}_a=\mathfrak{M}\cap \mathfrak{L}_a$.
If $\mathfrak{M}$ is also an ideal in $\mathfrak{L}$, then
$\mathfrak{L}/\mathfrak{M}$ is graded with $(\mathfrak{L}/\mathfrak{M})_a :=
\mathfrak{L}_a/\mathfrak{M}_a$.

Fix a Lie algebra $\mathfrak{L}$ that is $L$-graded for some lattice $L$. 
An $L$-graded module is an 
$\mathfrak{L}$-module $V$  such that 
$$V =\bigoplus_{a\in L} V_a$$
where $\mathfrak{L}_a\cdot V_b\subseteq V_{a+b}$ and  $V_a$ is finite dimensional for all $a,b\in L$.

Each component $\mathfrak{L}_a$ has the topology of a finite dimensional $\C$-vector space, which induces the product topology 
on $\mathfrak{L}=\bigoplus_{a\in L}\mathfrak{L}_a\subseteq \prod_{a\in L}\mathfrak{L}_a$.
We get a similar topology on $V$.
This in turn induces topologies on $\End_{\mathfrak{L}}(V)$ and $\Aut_{\mathfrak{L}}(V)$. 
All of these topological spaces are sequential~\cite{Munkres}, so we can use limits of sequences or series to prove our topological results.

Suppose $M$ is another lattice and  $f: L\to M$ is a homomorphism with the property that, for each $m\in M$, 
$\mathfrak{L}_a=0$ for all but finitely many $a\in f^{-1}(m)$.
Then we get an $M$-grading 
$\mathfrak{L} =\bigoplus_{m\in M} \mathfrak{L}_m$ by setting
$$ \mathfrak{L}_{m} = \bigoplus_{a\in f^{-1}(m)} \mathfrak{L}_a,$$
for all $m\in M$. 
Similarly, if $V_a=0$ for all but finitely many $a\in f^{-1}(m)$, then we get an $M$-grading of $V$ with $ V_{m} = \bigoplus_{a\in f^{-1}(m)} V_a$.

\subsection{Gradings of $\fg$ and $V$}\label{SS-Vgraded}
Now consider  Kac--Moody algebra $\fg$ and weight module $V$ with $\wts(V)\subseteq P$.
Then $\fg$ is graded by the root lattice $Q$. Also
$V=\bigoplus_{\mu\in L_V} V_\mu$ where $L_V=\Span_\Z(\wts(V))$.
Thus we can consider $\fg$ and $V$ to be graded over $M_V:=\Span_\Z(Q, L_V)$. 
By \cite[Lemma 27]{Steinberg},  $V$ is faithful if and only if $Q\subseteq L_V$  if and only if $M_V=L_V$.

We find it convenient to have a $\Z$-grading for $\fg$ and $V$, so we need an appropriate function $f:M_V\rightarrow\Z$.
We now construct a map $f$ that works for a large class of weight modules (see Lemma~\ref{L-f}).
We start with the usual height function $\htt:\Delta\rightarrow\Z$ where $\Delta$ is the set of roots.
This extends to a $\Q$-linear map $f':\Span_\Q(\Delta)\rightarrow\Q$ defined by
$$ f'(\a)=\sum_{i\in I} a_i$$
where $\a=\sum_{i\in I} a_i\a_i$.
Now extend $f'$ to a $\Q$-linear map $\Span_\Q( P)\rightarrow\Q$ (the choice of extension is not important).
Now let $q_j=f'(b_j)$ for some basis $\{b_j\}$ of $M_V$.
We can choose a positive integer $a$ so that $aq_j\in\Z$ for all $j$.
Now take $f$ to be the restriction of $af'$ to $P$.
Note that $f(\a)=a\htt(\a)$ for $\a\in\Delta$.
If $L_V\subseteq Q$, then we can take $a=1$.
When $\det(A)\ne0$, we can take $a$ to be a divisor of $|\!\det(A)|$.

We now write
$$ \Delta_m:= \{\a\in\Delta\cup\{0\}\mid f(\a)=m\},  \qquad \wts_m(V):=\{\mu\in\wts(V)\mid f(\mu)=m\}.$$
If $V$ is the adjoint representation, then we can take $f=\htt$ and $\wts_m(V)=\Delta_m$.
\begin{lemma}\label{L-f}
Suppose $V$ is in Category $\mathcal{O}$, is in opposite Category $\mathcal{O}$, or is the adjoint module.
Then the map $f$ as above satisfies the conditions:
\begin{itemize}
\labitem{\rm (F1)}{L-f-lin}
$f$ is $\Z$-linear,
\labitem{\rm (F2)}{L-f-fin} $\Delta_m$ and $\wts_m(V)$ are finite sets for every $m\in\Z$,
\labitem{\rm (F3)}{L-f-Delta} $\Delta_m\subseteq\Delta^+$ if $m>0$, $\Delta_0=\{0\}$, and $\Delta_m\subseteq\Delta^-$ if $m<0$, 
\labitem{\rm (F4)}{L-f-incomp} the weights in $\wts_m(V)$ are pairwise incomparable under the dominance ordering for every $m\in\Z$.
\end{itemize}
\end{lemma}
Note that \ref{L-f-lin} and \ref{L-f-fin} are the conditions from Subsection~\ref{SS-graded} for $\fg$ and $V$ to be $\Z$-graded with
$$
 \fg_m:= \bigoplus_{\a\in\Delta_m} \fg_\a \quad\text{and} \quad V_m:=\bigoplus_{\mu\in\wts_m(V)} V_\mu
$$
for all $m\in\Z$.
If $V$ is a weight module with $\wts(V)\subseteq P$ and a function $f:P\to\Z$ satisfying \ref{L-f-lin}--\ref{L-f-incomp},
we say that $V$ is  a \emph{standard graded module}.
We do not know whether every weight module  with $\wts(V)\subseteq P$ is a standard graded module.

\subsection{Completions of $\fg$ and $V$}
Let $\fg$ be a Kac--Moody algebra  and let $V$ be a standard graded module for $\fg$.
Note that 
$$\fh=\fg_0, \qquad \fn^{+} = \bigoplus_{m>0} \fg_{m},\qquad \fn^{-} = \bigoplus_{m<0} \fg_{ m}.$$
For $m\in\Z$, define 
$${\fn}_{m}:= \bigoplus_{k\ge m} \fg_{k}  \quad\text{and}\quad
{N}_{m}:= \bigoplus_{k\ge m} V_{k}.$$
Note that each ${N}_{k}$ is an ${\mathfrak n}^+$-module, and
${\fn}_{k}$ is a Lie algebra for $k\ge0$ but only an ${\mathfrak n}^+$-module  for $k<0$.
We have a descending chain of ideals 
$$\mathfrak{n}^+ = \mathfrak{n}_1 \ge \mathfrak{n}_2 \ge \cdots\ge \mathfrak{n}_i\ge \cdots. $$
For all $k\ge1$, $\mathfrak{n}^+/\mathfrak{n}_k$ is finite dimensional by Lemma~\ref{L-f}\ref{L-f-fin},  
and $\mathfrak{n}_k/\mathfrak{n}_{k+1}$ is central in  $\mathfrak{n}^+/\mathfrak{n}_{k+1}$ by Lemma~\ref{L-f}\ref{L-f-incomp}. 
Since $\bigcap_{k=1}^\infty\fn_k=0$, $\fn^+$ is residually  nilpotent (see \cite{Ku}). 

We define the \emph{(positive formal) completions} of $\frak g$ and $V$ to be
$$\widehat{\fg} = \bigoplus_{m<0}\fg_m \oplus \prod_{m\ge0} {\fg}_m 
\quad\text{and}\quad
\widehat{V}=\bigoplus_{m<0}V_m \oplus \prod_{m\ge0}V_m$$
respectively.
Also define
$$\widehat{\fn}^+_{m}:= \prod_{k\ge m} \fg_{k}  \quad\text{and}\quad
\widehat{N}_{m}:= \prod_{k\ge m} V_{k}$$
for $m\in\Z$, and let $\widehat{\frak n}^+ := \widehat{\frak n}^+_1$.
Again note that each $\wh{N}_{m}$ is an 
 $\wh{\mathfrak n}^+$-module, and
$\wh{\fn}_{m}$ is a Lie algebra for $m\ge0$, but only a $\wh{\mathfrak n}^+$-module for $m<0$.
Note that $\widehat{V}=V$ if the set $f(\wts(V))$ is bounded above.
In particular, this  holds for all Category $\mathcal{O}$ modules, but does not hold for opposite Category $\mathcal{O}$ modules.
When $V$ is the adjoint module of $\fg$, $\widehat{V}$ is naturally isomorphic to the adjoint module of~$\wh\fg$.
We extend $\rho$ from a representation of $\fg$ acting on $V$ to a representation  of $\widehat{\fg}$ acting on $\widehat{V}$.

We have a descending chain of ideals 
$$\wh{\fn}^+ = 
\wh{\fn}_1 \ge \wh{\fn}_2 \ge \cdots\ge \wh{\fn}_k \ge \cdots $$
such that $\wh{\fn}/\wh{\fn}_k$ is isomorphic to $\fn/\fn_k$ for $k\ge0$. 
Thus
$$\widehat{\frak n}^+\cong \varprojlim_{i\geq 1}{\mathfrak{n}}^+/{\mathfrak{n}_i},$$
the \emph{pro-nilpotent completion} of $\mathfrak{n}^+$.
So $\widehat{\frak n}^+$ is pro-nilpotent and $\fn^+\subseteq\wh{\fn}^+$.
Hence our definition of $\wh\fg$ agrees with the definition
$$\widehat{\frak g} = \frak n^- \oplus \frak h \oplus \widehat{\frak n}^+$$
from \cite[Section~IV.4]{Ku}.
Also, for each $m\in\Z$, $\widehat{N}_{m}$ is a pro-module for $\widehat{\mathfrak n}^+$ in the sense of \cite[Definition~IV.4.23]{Ku}.

We  note that $\widehat{V}$ can be considered as a $\fg$-module or as a $\widehat\fg$-module.
If $V$ is integrable, then
$G_V\subseteq \GL(V)\subseteq \GL(\widehat{V})$.
In the case of the adjoint representation $V=\fg$, we have $G_{\ad}:=G_\fg \subseteq\Aut({\fg})\subseteq\Aut(\widehat{\fg})$.


\subsection{Pro-summable maps}\label{SS-pro-sum}
The completions $\wh\fg$ and $\wh{V}$ inherit the product topology from  $\prod_{m\in\Z}\fg_m$ and $\prod_{m\in\Z}V_m$, which  induces topologies on $\Aut(\wh\fg)$ and $\GL(\wh{V})$, which are also sequential topological spaces. 


Write $\pi_{m}$ for the projection $\wh{V}\rightarrow V_m$.
For $v\in \widehat{V}$, write ${v}_m:=\pi_m v\in {V_m}$  and let $N=N(v)$ be the \emph{order} of~$v$, that is, the smallest integer with $v_{N}\ne0$. 
Then $v\in \widehat{V}$ has a unique expression
$$v=\sum_{m=N}^\infty v_m.$$
Similarly write $\pi_m$ for the projection $\wh{\fg}\rightarrow \fg_m$, so that each $x\in\widehat{\fg}$ has a unique expression $x=\sum_{m=N}^\infty x_m$ for $N=N(x)$ and $x_m=\pi_m x$.

Let $J$ be a countable index set and let $\{\vartheta_j\}_{j\in J}$ be a collection of linear operators $\wh{V}\rightarrow\wh{V}$ such that $\vartheta_j(\wh{N}_k)\subseteq \wh{N}_k$ for all $j\in J, k\in\Z$.
Let $\overline\vartheta_{j,k}$ be the induced map
$\wh{V}/\wh{N}_k\rightarrow\wh{V}/\wh{N}_k$.
We say that $\{\vartheta_j\}_{j\in J}$ is \emph{pro-summable} if   $\{\overline\vartheta_{j,k}\}_{j\in J}$ is summable for all $k\in\Z$.
In other words, for all $v\in \wh{V}$ and $k\in\Z$, the number of $m\le k$ for which $\pi_m(\vartheta_k(v))\ne0$ is finite.
Summability ensures that $\sum_{j\in J}\bar\vartheta_{j,k}$ is a well-defined endomorphism on $\wh{V}/\wh{N}_k$.
Hence we can define an endomorphism on $\wh{V}$ by
$$\sum_{j\in J}\vartheta_j:= \varprojlim_{k\geq 1} \sum_{j\in J}\vartheta_{j,k}.$$
More concretely, 
$$ \pi_m \sum_{j\in J}\vartheta_j(v) := \sum_{j\in J} \pi_m \vartheta_j(v) $$
for all $v\in \wh{V}$.
Note that $\wh{V}/\wh{N}_k\cong
V/N_k$, which is a vector space of finite or countable dimension.
(For a more general definition of pro-summability see \cite[Subsection~3.1]{CJM}).

For convenience, we read the product symbol  from the right to the left. For example,
\begin{align*}
\prod_{k=1}^\infty \exp\left(\rho\left(x_k \right) \right)
&:= \cdots  \exp\left(\rho\left(x_3 \right) \right)  \exp\left(\rho\left(x_2 \right) \right) \exp\left(\rho\left(x_1 \right) \right)\\
&= \lim_{n\to\infty}\,  \exp\left(\rho\left(x_n \right) \right) \cdots  \exp\left(\rho\left(x_2 \right) \right) \exp\left(\rho\left(x_1 \right) \right).
\end{align*}
The following result is now straightforward:
\begin{proposition}\label{P-formulas}
Let $x=\sum_{k=1}^\infty x_k\in\wh\n^+$ for $x_k\in \frak g_k$  for all $k\geq1$. 
Then $\exp(\rho(x))$ and $\prod_{k=n}^\infty \exp\left(\rho\left(x_k\right)\right)$ expand to pro-summable series given by the formulas
\begin{align}
\pi_k\left(\exp(\rho(x)) \cdot v\right)   &= v_k+\sum_{n=1}^{k-N(v)} \sum_{\substack{m,n_1,\dots,n_m>0,\\ j_1,\dots,j_m>0,\text{ s.t.}\\n_1j_1+\cdots+n_mj_m=n}}
\frac{x_{j_m}^{n_m} \cdots x_{j_1}^{n_1}}{n_1!\cdots n_m!}  \cdot v_{k-n},\label{E-x}\\
\pi_k\left(\prod_{j=n}^\infty \exp\left(\rho\left(x_j \right) \right)\cdot v\right)&= v_k+\sum_{n=1}^{k-N(v)} \sum_{\substack{m,n_1,\dots,n_m>0,\\ 0<j_1<\dots<j_m,\text{ s.t.}\\n_1j_1+\cdots+n_mj_m=n}}
\frac{x_{j_m}^{n_m} \cdots x_{j_1}^{n_1}}{n_1!\cdots n_m!}  \cdot v_{k-n},\label{E-sumx}
\end{align}
for $x=\sum_{j=1}^\infty x_j$ and $v=\sum_{i=N(v)}^\infty v_i$.
Hence they define elements of $\GL(\wh{V})$.
In particular, $\exp\left(\ad\left(x \right) \right)$ and $\prod_{k=1}^\infty \exp\left(\ad\left(x_k \right) \right)$
define automorphisms in $\Aut\left(\wh{\fg}\right)$.
\end{proposition}
\begin{proof}
The formulas are straightforward and show the series are pro-summable.
The fact that these maps are automorphisms in the adjoint case is a standard property of exponential series.
\end{proof}

Note that $\wh\fn^+$ is not generally commutative, so $\prod_{k=1}^\infty \exp\left(\rho\left(x_k \right) \right)$ is not generally equal to 
$\exp\left(\rho\left(\sum_{k=1}^\infty x_k \right) \right)$.

    
\begin{lemma} \label{L-cont}
The maps $\wh\fn^+\rightarrow \GL(\wh{V})$ given by
$$x\mapsto \exp(\rho(x))\quad\text{and}\quad
x=\sum_{k=1}^\infty x_k\mapsto 
\prod_{k=1}^\infty \exp(\rho(x_k))
$$
are continuous.
\end{lemma}
\begin{proof}
The formulas given in Proposition~\ref{P-formulas} show that these maps are continuous when restricted to the subspace $\fg_k$ for $k\ge1$. So the maps are continuous on $\wh\fn^+=\prod_{k=1}^\infty \fg_k$.
\end{proof}

We can now extend the definition of $U_V^+$ to the case where $V$ is a non-integrable standard graded module:
$$ U_V^+ := \langle \chi_{\a_i}(t)\mid i\in I,\, t\in\C\rangle \subseteq \GL(\wh{V}).$$
The definition of $H_V$ also extends to this case, however we cannot extend the definition of $U^-_V$ or $G_V$ since $\wh{V}$ is a \emph{positive} completion.

\section{Complete groups}\label{prounipgp}
As above, $\fg$ is a Kac--Moody algebra with symmetrizable Cartan matrix $A$.
Let $V$ be a  standard graded module  with grading function $f:P\rightarrow\Z$.
Recall that $V=\sum_{j\in\Z}V_j$,
$\wh{V} = \sum_{j<0}V_j\oplus \prod_{j\ge0}V_j$,
and $\wh{N}_k =\prod_{j\ge k}V_j\subseteq \wh{V}$.

We now introduce some useful terminology from rigorous formulations of quantum field theory (see for example \cite{GJ}):
\begin{defn} Let $m$ be an integer and let $\varphi$ be a linear operator $\wh{V}\to\wh{V}$.
\begin{enumerate}
\item If $\varphi$ satisfies $\varphi(V_k)\subseteq V_{k+m}$ for all $k\in\Z$, we call $\varphi$ an \emph{$m$-shift operator}.
\item If $\varphi$ satisfies $\varphi(v)\in v+\wt{N}_{k+1}$ for all $k\in\Z$ and $v\in V_k$, we call $\varphi$ a \emph{smear operator}.
\end{enumerate}
\end{defn}

Note that $\rho(x_m)$ for $x_m\in\fg_m$ is an $m$-shift operator, and
$\exp(\rho(x))$ for $x\in\wh{\fn}^+$ is a smear operator.

Let $\varphi$ be a smear operator and set $\varphi_{jm}:= \pi_{j+m}\circ \varphi|_{\wh{V}_j} : V_j\to V_{j+m}$.
Then we can define an $m$-shift operator 
$$\varphi_m:\quad \wh{V}\to\wh{V}, \quad v:= \sum_{j=N}^\infty v_j \mapsto \sum_{j=N}^\infty \varphi_{jm}(v_j).$$
It is now  easily seen that $\varphi_m=0$ for $m<0$, $\varphi_0=1_{\wh{V}}$, $\{\varphi_m\}_m$ is pro-summable, and
$$\varphi = \sum_{m\in\Z}\phi_m = 1+\sum_{m=1}^\infty \varphi_m.$$
That is, every smear operator can be written as a pro-summable series of shift operators with leading term~1.

\subsection{The complete pro-unipotent group}
We define the \emph{complete pro-unipotent group} $\wh{U}^+_V$ to be the closure of $U^+_V$ in $\GL(\wh{V})$. This definition is intuitive but not explicit, so first we consider the groups
$$\widehat{{U}}^+_{V,n}=\left\{\prod_{k=n}^\infty \exp\left(\rho\left(x_k \right) \right)\mid x_k\in \fg_k \right\}$$
for $n$ a positive integer.
We will eventually prove that $\wh{U}^+_V$ is identical to $\wh{U}^+_{V,1}$ for $V$ faithful (Corollary~\ref{C-completion}).
Note that 
$\widehat{U}^+_{V,1}\supseteq \widehat{U}^+_{V,2}\supseteq\cdots$.

The following standard result follows formally from the  Baker--Campbell--Hausdorff  formula (see, for example, \cite[p.\ 32]{Stern}):
\begin{lemma}\label{BCH}
If $x,y\in \wh{\fn}^+$, then
$$ \exp(\rho(x))\exp (\rho(y)) = \exp(\rho(z))$$
for 
$$z= x+y+\frac{1}{2}[x,y]+\frac{1}{12}\left([x,[x,y]]-[y,[x,y]]\right)+\cdots \in \wh\fn^+$$
with all coefficients in $\Q$.
\end{lemma}

The proof of the following lemma describing  $\Uhp_{V, n}$ and its cosets  is now similar to the proof of Lemma~3.11 in~\cite{CJM}. 
\begin{proposition}
    \label{L-nilpexp}   For positive integers $n$ and $m$,  we have: 
\begin{enumerate}
\item\label{L-nilpexp-Uhat}  $\Uhp_{V,n}=\exp\left(\rho\left(\wh{\fn}^+_n\right)\right)$.
\item\label{L-nilpexp-coset}  $\exp\left(\rho(x)\right) \Uhp_{V,n+1} = \exp\left(\rho(x + \widehat{\fn}^+_{n+1})\right)$,  for $x\in\fg_n$.
\item\label{L-nilpexp-sum}  $\exp\left(\rho(x+y)\right) \Uhp_{V,n+1} = \exp\left(\rho(x))\exp(\rho(y)\right) \Uhp_{V,n+1}$,   for $x,y\in\fg_n$.
\item\label{L-nilpexp-prod}   $\exp\left(\rho\left([x,y]\right)\right)\Uhp_{V,n+m+1} = \left( \exp\left(\rho(x)\right), \exp\left(\rho(y)\right) \right)\Uhp_{V,n+m+1}$,  for $x\in\fg_n$ and $y\in\fg_m$.
\end{enumerate}
\end{proposition}
\begin{proof}
Suppose $x\in\widehat{\fn}^+_n$. Take  $x_n\in\fg_n$ with $x-x_n\in\widehat{\frak n}^+_{n+1}$.
So Lemma~\ref{BCH} gives us  $$x^{(n+1)} = x-x_n+ \frac{1}{2}\left[x,-x_n\right]+\cdots\in\widehat{\frak n}^+_{n+1}$$ such that
$$\exp\left(\rho\left(x\right)\right) \exp\left(\rho\left(-x_n\right)\right)= \exp\left(\rho\left(x^{(n+1)}\right)\right)$$
and so
$$\exp\left(\rho\left(x\right)\right)= \exp\left(\rho\left(x^{(n+1)}\right)\right) \exp\left(\rho\left(x_n\right)\right).$$
By induction we get
$$\exp\left(\rho(x)\right)= \exp\left(\rho\left(x^{(N+1)}\right)\right) \prod_{\ell=n}^N\exp\left(\rho\left(x_l\right)\right) $$
with $x^{(N+1)}\in\widehat{\frak n}^+_{N+1}$.
From the definition of the product topology it is clear that 
$$
\lim_{N\to\infty}\exp\left(\rho\left(x^{(N+1)}\right)\right)=1,$$
and so $\exp\left(\rho(x)\right)=\prod_{\ell=n}^\infty\exp\left(\rho\left(x_\ell\right)\right)\in \Uhp_n.$
Thus $\exp\left(\rho\left(\widehat{\frak n}^+_n\right)\right) \subseteq \Uhp_n$ and \eqref{L-nilpexp-Uhat} is proved. 

Parts~\eqref{L-nilpexp-coset} and \eqref{L-nilpexp-sum} are now straightforward, and (\ref{L-nilpexp-prod}) follows from Lemma~\ref{BCH}.
\end{proof}
\begin{corollary} $\exp\circ\rho$ induces  a group isomorphism between the additive group $\widehat{\frak n}^+_{n}/\widehat{\frak n}^+_{n+1}$ and the multiplicative group $\Uhp_{V,n}/\Uhp_{V,n+1}$.
\end{corollary}
\begin{proof} Proposition~\ref{L-nilpexp}(\ref{L-nilpexp-sum}) and~(\ref{L-nilpexp-prod}).\end{proof}

\begin{corollary}
$\displaystyle \bigcap_{k=1}^\infty \wh{U}_{V,k}=1$, so $\wh{U}_{V,1}$ is residually nilpotent.
\end{corollary}
\begin{proof} Use Proposition~\ref{L-nilpexp}(\ref{L-nilpexp-Uhat}).
\end{proof}

Given $\a,\b\in\Delta$, we define the \emph{root span} to be  $\Sigma(\a,\b)=\bigcup_{n=1}^\infty \Sigma_n$, where $\Sigma_n$ is defined inductively:
\begin{align*}
 \Sigma_1&= \{\a+\b\}\cap\Delta; &
 \Sigma_{n} &= \left\{\g+\a,\g+\b \mid \g\in\Sigma_{n-1} \right\}\cap\Delta.
\end{align*}
Note that $\Sigma(\a,\b)=\emptyset$ if and only if $\a+\b\notin\Delta$,
and $\Sigma(\a,\b)\subseteq \left\{a\a+b\b \mid a,b\in\Z_{\ge0} \right\}\cap\Delta$. 

\begin{proposition}\label{T-commrel2}
Suppose $\a,\b\in\Delta^+$, $x_{\a}\in\fg_{\a}$, $y_{\b}\in\fg_{\b}$. 
Then there are unique $z_{\g}\in\fg_{\g}$ for each
$\g\in\Sigma(\a,\b)$ such that
$$\left(\exp\left(u\rho\left(x_{\a}\right)\right),\exp\left(v\rho\left(y_{\b}\right)\right)\right) =\prod_{\g=i\a+j\b\in\Sigma(\a,\b)} \exp\left(u^i v^j\rho\left(z_{\g}\right)\right)$$
holds in $\wh{U}^+_{V,1}$ for every $u,v\in\C$.
In particular, $z_{\a+\b}=\left[x_{\a},y_{\b}\right]$.
\end{proposition}

\begin{theorem}\label{T-pro-unip} 
$\wh{U}_{V,1}^+$ is a pro-unipotent group  since
$\widehat{U}^+_{V,1}\cong \varprojlim_{i\geq 1}\wh{U}_{V,1}^+/{\wh{U}}_{V,i}$.
\end{theorem}
\begin{proof}
Let $k\ge n\ge1$. It is easily shown that 
$\wh{U}_{V,n}^+\cdot \wh{N}_k^+\subseteq \wh{N}_k$.
Hence $\wh{U}_{V,1}/\wh{U}_{V,k}$ acts faithfully on $\wh{V}/\wh{N}_k$ as a unipotent group. It now follows that $\wh{U}_{V,1}^+$ is pro-unipotent.
We have
\begin{align*} \prod_{k=1}^\infty \exp\left(\rho\left(x_k \right) \right) &=\varprojlim_{i} \prod_{k=1}^i \exp\left(\rho\left(x_k \right) \right) \wh{U}_{V,1}^+ .\qedhere
\end{align*}
\end{proof}


\begin{corollary}
$\widehat{{U}}^+_{V,n}$ is a closed subgroup of $\widehat{{U}}^+_{V,1}$.
\end{corollary}
\begin{proof} By the construction of the pro-unipotent group $\wh{U}^+_{V,1}$, the subgroups $\wh{U}^+_{V,n}$ are closed in the pro-unipotent topology on $\wh{U}^+_{V,1}$, which is identical to the original topology on $\wh{U}^+_{V,1}\subseteq\GL(\wh{V})$.
\end{proof}

For the rest of this section, we assume that $V$ is faithful.
\begin{lemma}\label{homeo}
If $V$ is faithful, then 
$x\mapsto \exp(\rho(x))$ and 
$
x=\sum_{k=n}^\infty x_k\mapsto 
\prod_{k=n}^\infty \exp(\rho(x_k))$
are both 
homeomorphisms
$\wh{\fn}^+_n\rightarrow\widehat{{U}}^+_{V,n}$.
\end{lemma}
\begin{proof}
Since $\fg$ acts faithfully of $V$, it follows that $\wh{\fg}$ acts faithfully on $\wh{V}$. Hence the
second map is one-to-one. It is onto by the definition of $\widehat{{U}}^+_{V,n}$. It is a homeomorphism because $\wh{\fn}^+_n$ and $\widehat{{U}}^+_{V,n}$ have the same product topology.
It is now straightforward to show that the first map is a homeomorphism using Proposition~\ref{P-formulas} and Proposition~\ref{L-nilpexp}.
\end{proof}

\begin{theorem}\label{P-gen}
Let $V$ be faithful and let $X$ be a subset of $\bigcup_{k=1}^\infty \fg_k$. 
Let $\wh{\mathfrak{s}}$ be the closure in $\wh\fn^+$ of the subalgebra generated by $X$ and let $\wh{S}$ be the closure in $\GL(\wh{V})$
of the group generated by $\{\exp(t\rho(x))\mid t\in\C,\,x\in X\}.$
Then $\wh{S}=\exp(\rho(\wh{\mathfrak{s}})).$
\end{theorem}
\begin{proof}
Let $E=\exp(\rho({\mathfrak{s}}))$,
where $\mathfrak{s}$ is the subalgebra generated by $X$.
By Lemma~\ref{homeo} for $n=1$,  $\exp \circ \rho$ is a homeomorphism $\wh{\fn}^+ \to \wh{U}^+_{V,1}$, so the closure of $E$ is $\wh{E}=\exp(\rho(\wh{\fs}))$.
Note that $\wh{E}$ is a group by Lemma~\ref{BCH}.

Let $S$ be the group generated by $ \{\exp(t\rho(x))\mid t\in\C,\,x\in X\}$. Then $\wh{S}$ is the closure of $S$ by definition.
For every $x \in X$ and $t\in\C$, we have $tx \in {\fs}$, and so $\exp(t\rho(x)) = \exp(\rho(tx)) \in E$.
Hence $S\subseteq E$ and so $\wh{S}\subseteq\wh{E}$.

It remains to show that $\wh{E} \subseteq \wh{S}$.
Let $y \in \wh{\fs}$, that is $y=\lim_{j\to\infty} y_j$ for some  $y_j$ in $\fs$. Then each $y_j$ can be written as an expression in elements of $X$, using scalar products, sums, and Lie products.
For any fixed positive integer $N$,
$\exp(\rho(y_j))\wh{U}^+_{V,N}$ is in $S\wh{U}^+_{V,N}$, by repeated application of Propositions~\ref{T-commrel2} and~\ref{L-nilpexp}
(note that this is a standard result for nilpotent groups \cite[Chapter~III, §6]{Bourbaki1-3} applied to $\wh{U}^+_{V,1}/\wh{U}^+_{V,N}$.)
Taking $N\to\infty$, we get $\exp(\rho(y_j))$ in $\wh{E}$. Since $\wh{E}$ is closed, $\exp\rho(y)=\lim_{j\to\infty} \exp\rho(y_j)\in \wh{E}$ as required.
\end{proof}






\begin{corollary}\label{C-completion}
For $V$ faithful,  $\widehat{U}^+_{V,1}=\wh{U}^+_V$, the topological closure of $U^+_V$ in $\GL(\wh{V})$.
\end{corollary}

Using Theorem~\ref{T-indep} and Corollary~\ref{C-completion} we obtain the following.

\begin{corollary}
For $V$ faithful, $\wh{U}_V$ is independent of the choice of $V$.
\end{corollary}
For $V$ faithful, we write $\wh{U}$ for $\wh{U}_V$.

\subsection{Other complete groups}
Suppose $V$ is faithful and integrable.
The group $\widehat{U}_V$ is the completion of the unipotent subgroup $U_V$ in the
Chevalley group $G_V$. 
Since $H_V$ and $U^-_V$ are subgroups of $\GL(\widehat{V})$, we can define
completions of the Borel and Chevalley groups as
$$\widehat{B}_V:= \langle H_V, \widehat{U}_V \rangle \quad\text{and}\quad \widehat{G}_V:= \langle U^{-}_V, \widehat{B}_V \rangle. $$
Similar complete groups have been constructed by Kumar \cite{Ku},  Marquis \cite{MarT} and Rousseau \cite{Rou}. 

Combining Corollary~\ref{C-completion} with  Theorem~\ref{T-indep} we obtain the following.
\begin{theorem}\label{T-indepG} 
For faithful integrable modules $V$, $\wh{G}_V$ depends only on~$L_V$.
\end{theorem}

For $V$ faithful and integral, we write $\wh{G}_L$ for $\wh{G}_V$ when $L=L_V$. When $V$ is the adjoint representation,  $\widehat{{U}}$ is just the set of all smear operators in $\Aut(\fg)$. 


\section{Root groups}\label{S-rtgp}
In this section, we define the root algebra $\mathfrak n_\a$ corresponding to an arbitrary root $\a$ to be the smallest subalgebra of $\mathfrak g$ containing~$\mathfrak g_\a$. We also define  $\widehat{\mathfrak n}_\a$ to be the smallest complete subalgebra of $\widehat{\mathfrak g}$ containing $\mathfrak g_\a$. Note that this differs from the definition of \cite{MarT} which defines a larger root algebra. We also define imaginary  root groups and their completions.


For convenience, we assume that $V$ is faithful and  integrable, so $G_L$ and $\wh{G}_L$  depend on the lattice $L=L_V$, while $U$ and $\wh{U}$ are independent of the choice of $V$.
Note that all the results in this section can be extended to faithful standard graded modules, if we restrict ourselves to considering positive roots.

\begin{defn}\label{rootalgebra}
Given a root $\a$ (real or imaginary), we make the following definitions:
\begin{enumerate}
\item The \emph{root algebra of $\a$} is the subalgebra  $\fn_\a\subseteq\fg $ generated by  $\fg_\a$.
\item The \emph{complete root algebra of $\a$} is the pro-nilpotent completion $\wh{\fn}_\a$ of the root algebra $\fn_\a$.
\item The \emph{complete root group of $\a$} is the closure $\wh{U}_\a$ in $\wh{G}$ of the subgroup generated by those $\exp(\rho(x))$ for $x\in\fg_\a$ that are well-defined as elements of $\wh{G}$. 
\end{enumerate}
\end{defn}

Since we are using the positive completion, $\exp(\rho(x))$ in (3) may not be well-defined in $\wh{G}$ for a negative imaginary root $\a$. This is discussed in detail in Subsection 5.3.
Note that
$$
\fn_\a\subseteq \bigoplus_{n\ge1}\fg_{n\a} \quad\text{and}\quad
\wh{\fn}_\a\subseteq \prod_{n\ge1}\fg_{n\a}. 
$$
These larger algebras are discussed  in \cite{MarT} and \cite{Ku}. See Theorem~\ref{P-nalpha}  for a characterization of the root algebra of a root  $\a$, Theorem~\ref{P-nhat} for a characterization of the complete root algebra of  $\a$, and Lemma~\ref{L-Uhat-neg} or Proposition~\ref{T-Uhat} for a characterization of the complete root group for $\a$.

It would seem natural to define the root group $\wh{U}_\a\subseteq\wh{G}$ for an imaginary root $\a$ as a subgroup of $G$. However,  the complete root group of $\a$ intersects trivially with $G$ for most of the roots $\a$ that we are interested in.

\subsection{Free groups and Magnus algebras}\label{FreeMagnus}
We review some results on free groups and associated structures. See~\cite{Bourbaki4-6} and ~\cite{Ba93} for details. Let $X=\{x_1,\dots,x_m\}$ be a finite set of symbols, and let $F(X)$ be the free group of $X$.

Let $\C\langle X\rangle$ be the algebra of polynomials in noncommuting variables.
If we endow $\C\langle X\rangle$ with the Lie product $[x,y]=xy-yx$, then it becomes a Lie algebra. The free Lie algebra $\mathfrak{F}(X)$ is the Lie subalgebra generated by $X$.  
Note that $\C\langle X\rangle$ is the universal enveloping algebra of $\mathfrak{F}(X)$ and is isomorphic to the free associative algebra.

Let
$N_k$ be defined inductively by $N_1:=F(X)$ and $N_k:=(N_{k-1},F(X))$.
The \emph{lower central series}
$${F}(X)= {N}_1 \ge {N}_2 \ge \cdots$$
does not terminate, that is, we never have $N_k=N_{k+1}$.

However $\bigcap_{k=1}^\infty N_k=\emptyset$ by \cite{Ba93}, so
$F(X)$ can be considered as a subgroup of its pronilpotent completion
$$\wh{F}(X)= \varprojlim_k F(X)/N_k.$$ 
Here $\wh{F}(X)$ is defined as  the projective limit of the algebraic nilpotent groups $F(X)/N_k$. 
Note that $F(X)$ is not a closed subgroup of  $\wh{F}(X)$ and is not itself pronilpotent. 

The free Lie algebra $\mathfrak{F}(X)$
is a $\Z$-graded Lie algebra with
$\mathfrak{F}(X)_n$ being the set of $n$-homogeneous Lie polynomials in $\C\langle X\rangle$ for $n\ge0$ and $\mathfrak{F}(X)_n=0$ for $n<0$.
Let 
$$\mathfrak{N}_k = \bigoplus_{k>n}^\infty \mathfrak{F}(X)_k.$$
Then $\mathfrak{N}_k$ coincides with the $k$-th subalgebra of the lower central series defined inductively by
$\mathfrak{N}_1 := \mathfrak{F}(X)$ and $ \mathfrak{N}_{n+1} := [\mathfrak{F}(X), \mathfrak{N}_n]$.
The pronilpotent completion of $\mathfrak{F}(X)$ is 
\[
\widehat{\mathfrak{F}}(X) = \varprojlim_n \mathfrak{F}(X)/\mathfrak N_n.
\]
As with the free group, $\mathfrak{F}(X)$ is a subalgebra of $\wh{\mathfrak{F}}(X)$, but $\mathfrak{F}(X)$ is not a closed subalgebra and is not itself pro-nilpotent.

The universal enveloping algebra $U(\widehat{\mathfrak{F}}(X))$ of the pronilpotent completion of the free Lie algebra is the algebra $\C\llangle X\rrangle$ of formal power series in non-commuting variables $X$. The algebra $\C\llangle X\rrangle$ is also called the Magnus algebra. The collection of power series with no constant term forms an ideal $\mathcal{I}(X)$ in $\C\llangle X\rrangle$.  The set $1+\mathcal{I}(X)$ forms a group called the Magnus group $M(X)$ \cite{Ba93}. This  is the group of  power series in $\C\llangle X\rrangle$ with constant term~1.


    

\subsection{Root subalgebras}
We start by describing root subalgebras in $\fg$ and their completions.
Let $\a\in\Delta$ and $m=\mult(\a)=\dim(\fg_\a)$.
Recursively define $\fg_{\a,1}:=\fg_{\a}$ and 
$$\fg_{\a,k}:= [\fg_\a,\fg_{\a,k-1}]=\Span\{ [x,y]  \mid x\in\fg_\a,\; y\in\fg_{\a,k-1}\}$$
for $k\ge2$.
\begin{lemma}\label{L-gak} Let $\a\in\Delta$.
\begin{enumerate}
\item\label{L-gak-cont} $\fg_{\a,k}\subseteq \fg_{k\a}\subseteq \fg_n$ where $n={kf(\a)}$.
\item\label{L-gak-add} $[\fg_{\a,j},\fg_{\a,k}]\subseteq \fg_{\a,j+k}$ for $j,k\ge1$
\item\label{L-gak-prod} A repeated Lie product of $x_1,\dots,x_k\in \fg_\a$, with any arrangement of the Lie brackets,
is in $\fg_{\a,k}$.
\end{enumerate}
\end{lemma} 
\begin{proof}
    (\ref{L-gak-cont}) is clear. (\ref{L-gak-prod}) can be proven by complete induction and the Jacobi identity. We prove (\ref{L-gak-add})  by induction on $j$. For $j=1$,
    $$[\fg_{\a,1},\fg_{\a,k}]=[\fg_{\a},\fg_{\a,k}]=\fg_{\a,k+1}$$
   and the statement  follows by  definition. Next, suppose
    $$[\fg_{\a,m},\fg_{\a,k}]=\fg_{\a,k+m}$$
    for all $k$. Then 
    \begin{eqnarray*}
[\fg_{\a,m+1},\fg_{\a,k}]&=&[[\fg_{\a}, \fg_{\a,m}],\fg_{\a,k}]\text{ by definition} \\
&=&-[\fg_{\a,k}, [\fg_{\a}, \fg_{\a,m}]]\\
&=&[ \fg_{\a},[\fg_{\a,m},\fg_{\a,k}]+[\fg_{\a,m},[\fg_{\a,k},\fg_{\a}]]\;\text{by the Jacobi Identity}.
    \end{eqnarray*}

By the inductive hypothesis $[\fg_{\a,m},\fg_{\a,k}]=\fg_{\a,k+m}$ and by the definition $[ \fg_{\a}, \fg_{\a,k+m}]=\fg_{\a,k+m+1}$. Similarly, by the defintion $[\fg_{\a,k},\fg_{\a}]=\fg_{\a,k+1}$ and by the inductive hypothesis $[\fg_{\a,m}, \fg_{\a,k+1}]=\fg_{\a,m+k+1}$. 

Thus  \begin{align*}
[\fg_{\a,m+1},\fg_{\a,k}]&=[ \fg_{\a},[\fg_{\a,m},\fg_{\a,k}]+[\fg_{\a,m},[\fg_{\a,k},\fg_{\a}]]\\
&=[ \fg_{\a},\fg_{\a,k+m}]+[\fg_{\a,m},\fg_{\a,k+1}]\\
&=\fg_{\a,m+k+1}+\fg_{\a,m+k+1}.
\end{align*}
So,
$$[\fg_{\a,m+1},\fg_{\a,k}]=\fg_{\a,m+k+1}.$$
This completes the induction and the equality follows.
\end{proof}

Let
$$\fn_{\a,n} := \bigoplus_{k=n}^\infty \fg_{\a,k}.$$
In general, when $\a$ is imaginary, $\fn_{\a,n}$ is a proper subspace of $\bigoplus_{k\geq n} \fg_{k\a}$.

The following proposition establishes that $\fn_{\a,1}$ coincides with the root algebra $\fn_\a$ of Definition~\ref{rootalgebra}.
\begin{theorem}\label{P-nalpha} 
Let $\a\in\Delta$ and let $b_1,\dots,b_m$ be a basis for $\fg_\a$.
\begin{enumerate}
\item\label{P-nalpha-sub} The root algebra $\fn_\a$ equals $\fn_{\a,1}$.
\item\label{P-nalpha-pro} $\fn_\a$ is pro-nilpotent with series
$$\fn_\a=\fn_{\a,1}\ge \fn_{\a,2} \ge\cdots.$$
\item\label{P-nalpha-real} If $\a$ is real, then $m=1$ and $\fn_\a=\fg_\a=\C b_1$.
\item\label{P-nalpha-isog} If $(\a,\a)=0$, then $\fn_\a=\fg_\a=\bigoplus_{i=1}^m\C b_i$ is abelian.
\item\label{P-nalpha-free} If $(\a,\a)<0$, then $\fn_\a$ is free over $\fg_\a$, that is, $\fn_\a$ is $\mathfrak{F}(b_1,\dots,b_m)$.
\item\label{P-nalpha-w} If $w\in W$, then $\wt{w}\cdot{\fn}_\a = {\fn}_{w\a}$.\\

\end{enumerate}
\end{theorem} 
\begin{proof}
    (\ref{P-nalpha-sub}) is immediate from Lemma~\ref{L-gak}(\ref{L-gak-prod}).  Lemma~\ref{L-gak}(\ref{L-gak-cont}) implies that $\fn_\a$ is a graded subalgebra of~$\fg$, and so (\ref{P-nalpha-pro}) follows.
    (\ref{P-nalpha-isog}) and~(\ref{P-nalpha-free}) follow from \cite{Kac90}, Corollary 9.12.
    (\ref{P-nalpha-w}) follows from~(\ref{P-nalpha-sub}) and Lemma~\ref{ww}.
    Since $\a=w\a_i$ for some $w\in W$ and $i\in I$, and $\fn_{\a_i}=\fg_{\a_i} =\C e_i$ by~(\ref{P-nalpha-w}), we have proved (\ref{P-nalpha-real}).
    \end{proof}

Let
$$
\wh\fn_{\a,n} := \prod_{k=n}^\infty \fg_{\a,k}.$$
The following proposition establishes that $\wh{\fn}_{\a,1}$ coincides with the root algebra $\wh{\fn}_\a$ of Definition~\ref{rootalgebra}.
Note that $\wh\fn_\a$ can be defined for every root $\a$, but it is not contained in $\wh\fg$ for every root because $\wh\fg$ is only complete in the positive direction.
\begin{theorem}\label{P-nhat}
Let $\a\in\Delta$ and let $b_1,\dots,b_m$ be a basis for $\fg_\a$.
\begin{enumerate}
\item\label{P-nhat-pronilp} $\wh\fn_{\a,1}$ is the pro-nilpotent completion of $\fn_\a$, and hence the complete root algebra $\wh\fn_{\a}=\wh\fn_{\a,1}$.
\item\label{P-nhat-closure} If $\a$ is positive or $(\a,\a)\ge0$, then $\wh\fn_\a$ is the closure in $\wh\fg$ of $\fn_\a$. 
\item\label{P-nhat-nonclosure} If $\a$ is negative and $(\a,\a)<0$, then $\wh\fn_\a$ is not contained in $\wh\fg$.
\item\label{P-nhat-abel} If $(\a,\a)\ge0$, then $\wh\fn_\a=\fn_\a$ is abelian.
\item\label{P-nhat-free} If $(\a,\a)<0$, then $\wh\fn_\a$ is 
$\wh{\mathfrak{F}}(b_1,\dots,b_m)$.
\item\label{P-nhat-w} Let $w\in W$. If $(\a,\a)\ge0$ or $\a\in\Delta^+$, then $\wt{w}\cdot \wh{\fn}_\a = \wh{\fn}_{w\a}$ in $\wh\fg$.
\end{enumerate}
\end{theorem}
\begin{proof}
    Since $\wh\fn/\wh\fn_k \cong \fn/\fn_k$, we have
    $\wh\fn_{\a, 1} = \varprojlim_{i\geq 1}\mathfrak{n}_{\a,1}/\mathfrak{n}_{\a,i},$  and so (\ref{P-nhat-pronilp}) follows from  Theorem~\ref{P-nalpha}(\ref{P-nalpha-pro}).
    (\ref{P-nhat-abel})~follows from Theorem~\ref{P-nalpha}(\ref{P-nalpha-real},\ref{P-nalpha-isog}).
    (\ref{P-nhat-closure})~follows from (\ref{P-nhat-abel}) when $(\a,\a)=0$, and from   Lemma~\ref{L-gak}(\ref{L-gak-cont})
    when $\a$ is positive.
    (\ref{P-nhat-free}) follows from (\ref{P-nhat-pronilp}), Theorem~\ref{P-nalpha}~(\ref{P-nalpha-free}) and~\cite{Ba93}.
    (\ref{P-nhat-nonclosure}) follows from Lemma~\ref{L-gak}(\ref{L-gak-cont}), using
    (\ref{P-nhat-free}) to show that each $\fg_{\a,k}$ is nontrivial.
    Recall that $W\cdot\Delta^{\im}_+=\Delta^{\im}_+$, and so $\wh\fn_\a\subseteq\wh\fg$ if and only if $\wh\fn_{w\a}\subseteq\wh\fg$, hence
    (\ref{P-nhat-w}) follows from Theorem~\ref{P-nalpha}~(\ref{P-nalpha-w}).    
\end{proof}





\subsection{Complete root groups}
The following lemma follows immediately from the fact that the group $\wh{G}$ is complete only in the positive direction.

\begin{lemma}\label{L-Uhat-neg}
If $\a$ is negative and $(\a,\a)<0$, then $\exp(\rho(x))$ is not well-defined as an element of $\wh{G}$ for any nonzero $x\in\fg_\a$.
Hence the  complete root group is trivial in this case.
\end{lemma}
Note that these root groups would be nontrivial if we took the \emph{negative} completions of $\fg$ and $G$.

Now assume $\a$ is positive or $(\a,\a)\ge0$, so that $\wh\fn_\a\subseteq\wh\fg$
by Theorem~\ref{P-nhat}(\ref{P-nhat-closure}).
Let
$$
\Uhat_{\a,n} := \left\{ \prod_{k=n}^\infty \exp(\rho(x_k)) \mid x_k\in\fg_{\a,k}\right\}.$$

Note that we have one-parameter subgroups
$$\exp(\rho(\C b_j)):= \{\exp(\rho(t b_j))\mid t\in\C\}\cong\C^+,$$
and write 
$$U_\a:= \left\langle \exp(\rho(\C b_j)) \mid j=1,\dots,m\right\rangle.$$
When $\a$ is real and  $m=1$, this is the usual definition of a real root group.
 If $\a$ is not real,  $U_\a$ may not be contained in $U^+$,  and in general is not independent of the choice of basis.  Hence this definition of $U_\a$ is not a suitable candidate for a root group when $\a$ is imaginary.

The following proposition establishes that $\wh{U}_{\a,1}$ coincides with the root group $\wh{U}_\a$ of Definition~\ref{rootalgebra}.
\begin{theorem}\label{T-Uhat} Let $\a\in\Delta$ and let $b_1,\dots,b_m$ be a basis for $\fg_\a$.
Suppose $\a$ is positive or $(\a,\a)\ge0$.
\begin{enumerate}
\item\label{T-Uhat-nhat} The root group  $\wh{U}_\a=\wh{U}_{\a,1}=\exp(\rho(\wh\fn_\a))$.
\item\label{T-Uhat-closure} $\wh{U}_\a$ is the closure in $\wh{U}$ of $U_\a$.
\item\label{T-Uhat-prounip} $\wh{U}_\a$ is a pro-unipotent completion of $U_\a$ with  series
$$\wh{U}_\a=\wh{U}_{\a,1}\ge \wh{U}_{\a,2} \ge\cdots.$$
\item\label{T-Uhat-abel} If $(\a,\a)\ge0$, then $\wh{U}_\a=U_\a$ is the direct product of  
$\exp(\rho(\C b_i))$ for $i=1,\dots,m$.
\item\label{T-Uhat-free} If $\a$ is positive and $(\a,\a)<0$, then $\wh{U}_\a$ is isomorphic to the Magnus group  $M(b_1,\dots,b_m)$. 

\item\label{CompleteProSum} The elements of $\wh{U}_\a$ are pro-summable series acting on $\widehat{V}$.

\item\label{T-Uhat-w} Let $w\in W$. Then $\wt{w} \wh{U}_\a \wt{w}^{-1}= \wh{U}_{w\a}$.
\end{enumerate}
\end{theorem}
\begin{proof}
    (\ref{T-Uhat-nhat}) and~(\ref{T-Uhat-closure}) follow from Proposition~\ref{P-gen}.
    (\ref{T-Uhat-prounip}) follows from Section~\ref{FreeMagnus}.
    (\ref{T-Uhat-abel}) follows from Theorem~\ref{P-nhat}(\ref{P-nhat-abel}).
    (\ref{T-Uhat-free}) follows from Theorem~\ref{P-nhat}(\ref{P-nhat-free}) and Section~\ref{FreeMagnus}. (\ref{CompleteProSum}) follows from Section~\ref{prounipgp}.
    (\ref{T-Uhat-w}) follows from~(\ref{T-Uhat-nhat}) and Theorem~\ref{P-nhat}(\ref{P-nhat-w}).
\end{proof}


\bibliographystyle{amsalpha}
\bibliography{GPmath}{}

\end{document}